\documentclass[a4paper]{article}
\pdfoutput=1

\usepackage{amsmath}
\usepackage{amsfonts, amssymb, amsthm}
\usepackage{mathtools}
\usepackage{microtype}
\usepackage{color}
\usepackage[a4paper]{geometry}
 
\theoremstyle{plain}
\newtheorem{satz}{Theorem}
\newtheorem*{satz*}{Theorem}

\theoremstyle{definition}

\theoremstyle{plain}

\theoremstyle{plain}

\newcommand{\un}[1]{\underline{#1}}
\newcommand{\reell}{\mathbb{R}}

\DeclareMathOperator{\diag}{diag}

\DeclareMathOperator{\rank}{rank}

\newcommand{\abs}[1]{\lvert #1 \rvert}
\newcommand{\transp}{\mathsf T}
\newcommand{\had}{{}\circ{}}
\newcommand{\ones}{\mathbf{1}}
\newcommand{\DK}{\mathcal{K}}
\newcommand{\DKt}{\mathcal{K}_\transp}

\newcommand{\FF}{\mathcal F}
\newcommand{\Lip}{\Lambda}

\newcommand{\sev}{\vartheta}
\newcommand{\const}{\delta}

\begin{document}

\title{\bfseries A note on overrelaxation in the \\ Sinkhorn algorithm}
\author{Tobias Lehmann\thanks{Universit\"at Leipzig, Fakult\"at f\"ur Mathematik und Informatik, Augustusplatz 10, 04109 Leipzig, Germany} \qquad Max-K. von Renesse${}^*$ \and  Alexander Sambale${}^*$ \qquad  Andr\'e Uschmajew\thanks{Max Planck Institute for Mathematics in the Sciences, 04103 Leipzig, Germany}}
\date{}

\maketitle

\begin{abstract}
We derive an a priori parameter range for overrelaxation of the Sinkhorn algorithm, which guarantees global convergence and a strictly faster asymptotic local convergence. Guided by the spectral analysis of the linearized problem we pursue a zero cost procedure to choose a near optimal relaxation parameter. 
\end{abstract}

\section{Introduction and statement of result}

The Sinkhorn algorithm is the benchmark approach to fast computation of the entropic regularization of optimal transportation~\cite{NIPS2013_4927}. Ultimately, one is faced with  the following numerical problem: Given two probability vectors $a \in \reell_{+}^m$, $b \in \reell^{n}_{+}$ and a matrix $K \in \reell^{m \times n}_{+}$, the goal is to find a pair of vectors $(u,v) \in \reell_{+}^m \times \reell_{+}^n$ such  that 
\begin{equation}\label{eq: problem P}
u \had Kv = a \quad \mbox{and} \quad v \had K^\transp u = b, 
\end{equation}
where $x \had y$ denotes the componentwise multiplication (Hadamard product) of vectors of equal dimension. Here $\reell_+$ refers to the positive reals. We assume $\min(m,n) \ge 2$.

In the standard Sinkhorn algorithm an approximating sequence $(u_\ell,v_\ell)$ starting from an initial vector  $v_0 \in \reell_+^n$ is constructed via the update rule  
 \[
 u_{\ell + 1} = \frac a {Kv_{\ell}}, \qquad v_{\ell+1} = \frac{b}{K^\transp u_{\ell+1}},
 \]
 where $\frac{x}{y}$ denotes the componentwise division of vectors of equal dimension. It is a classic result by Sinkhorn~\cite{Sinkhorn1967} that for any initial point $v_0 \in \reell^n_+$ the algorithm converges to a solution $(u^*,v^*)$ of~\eqref{eq: problem P}, which is unique modulo rescaling $(t u^*, t^{-1} v^* )$, $t >0$. Moreover, the convergence, e.g. of suitably normalized iterates $u_\ell / \| u_\ell \|$ and $v_\ell / \| v_\ell \|$, or using other equivalent distance measures like the Hilbert metric, is R-linear with an asymptotic rate at least $\Lip(K)^2$, where $\Lip(K) < 1$ is the Birkhoff contraction ratio defined in~\eqref{def_lambda} further below~\cite{FranklinLorenz89}. See also~\cite{peyre_cuturi2018} for an overview.
 
In this note we discuss a modified version of the Sinkhorn algorithm employing relaxation, which was recently proposed in~\cite{thibault17} and~\cite{Peyreetal2019}. It uses the update rule
 \begin{equation}\label{eq: modified algorithm}
 u_{\ell + 1} 
= u_\ell^{1-\omega} \had  \left(\frac{a}{ K v_\ell }\right)^{\omega}, \qquad v_{\ell + 1} 
= v_\ell^{1-\omega}  \had \left(\frac{b}{ K^\transp u_{\ell + 1} }\right)^{\omega},
\end{equation}
where $\omega > 0 $ is are suitably chosen relaxation parameter, and exponentiation is understood componentwise. In a log-domain formulation such as~\eqref{eq: iteration in cds} further below, the relation to the classic concept of relaxation in (nonlinear) fixed point iterations will become immediately apparent. Note that the iteration~\eqref{eq: modified algorithm} still has the solution of~\eqref{eq: problem P} as its unique (modulo scaling) fixed point. As illustrated in~\cite{thibault17} and~\cite{Peyreetal2019}, choosing the parameter $\omega$ larger than one can significantly accelerate the convergence speed compared to the standard Sinkhorn method, which sometimes can be slow. For optimal transport, such an improvement could be in particular relevant in the regime of small regularization, or when a high target precision is needed, such as in applications in density functional theory~\cite{Cotar2013}.

While global convergence for $\omega \neq 1$ is not obvious anymore, local convergence of the modified method is ensured for all $0 < \omega < 2$, and the asymptotically optimal relaxation parameter can be determined from its linearization at a fixed point $(u_*,v_*)$. In logarithmic coordinates, the linearization of the standard Sinkhorn method has the iteration matrix
\begin{equation}\label{eq: matrix M}
M = \diag\left(\frac 1 a \right) P_*^{} \diag\left( \frac 1 b \right) P_*^\transp, \qquad \text{where} \quad P_* = \diag(u^*) K \diag(v^*).
\end{equation}
The local convergence rate equals the second largest eigenvalue
\[
0 \le \sev^2 < 1
\]
of that matrix; see~\cite{Knight2008}. Note that $M$ has real and nonnegative eigenvalues since it is similar to a positive semidefinite matrix, and its largest eigenvalue equals one (the eigenvector having constant entries), which accounts for the scaling indeterminacy in the problem formulation. For the modified method with relaxation, the local rate is also related to $\sev^2$, which has been worked out in~\cite{thibault17} and is summarized in the following theorem. For convenience, we provide a brief outline how this result can be obtained at the end of section~\ref{sec: compositional data space}.

\begin{satz}[cf.~\cite{thibault17}]\label{thm: local rate modified}
Assume $\sev^2 > 0$. For all choices of $0< \omega < 2$ the modified Sinkhorn algorithm~\eqref{eq: modified algorithm} is locally convergent in some neighborhood of $(u^*,v^*)$. Its asymptotic (R-linear) convergence rate is
\begin{equation}\label{eq: definition r_alpha}
\rho_\sev(\omega) \coloneqq \begin{cases} \frac{1}{4}\left(\omega \sev + \sqrt{\omega^2 \sev^2 - 4(\omega -1 )}  \right)^2, \quad &\text{if $0 < \omega \le \omega^{\text{\upshape opt}}$,} \\ \omega - 1, \quad &\text{if $\omega^{\text{\upshape opt}} \le \omega < 2$}, \end{cases}
\end{equation}
where
\begin{equation}\label{eq: omegaopt}
\omega^{\text{\upshape opt}} = \frac{2}{1+\sqrt{1 - \sev^2}} > 1.
\end{equation}
It holds $\rho_\sev(\omega) < 1$ for all $0 < \omega < 2$, and $\omega^{\text{\upshape opt}}$ provides the minimal possible rate (independent of the starting point) on that interval, namely
\[
\rho^{\text{\text{\upshape opt}}} = \omega^{\text{\upshape opt}} - 1 = \frac{1 - \sqrt{1 - \sev^2}}{1 + \sqrt{1 - \sev^2}} < \sev^2.
\]
\end{satz}

By the above theorem, the optimal relaxation parameter $\omega^{\text{\upshape opt}}$ is always larger than one (if $\sev^2 > 0$). In fact, by the exact formula~\eqref{eq: definition r_alpha} for the convergence rate, the range of $\omega$ for which the modified method is asymptotically strictly faster than the standard Sinkhorn method, that is, $\rho_{\vartheta}(\omega) < \vartheta^2 = \rho_{\vartheta}(1)$, is precisely the interval
\begin{equation}\label{eq: feasible alphas}
1 < \omega < 1 + \sev^2.
\end{equation}
However, the value of $\sev^2$ depends on the solution and is therefore not known in advance. To deal with this problem, an adaptive procedure for choosing $\omega$ is proposed in~\cite{thibault17}.

As our contribution, the main goal in this note is to provide an a priori interval for the relaxation parameter $\omega$ for which the modified iteration is both globally convergent and locally faster than the standard Sinkhorn method. In Theorem~\ref{thm: global convergence} we first prove global convergence of the modified method for parameters in the interval $0 < \omega < \frac 2 { 1+\Lip(K)}$. In Theorem~\ref{thm: estimate for mu} we then provide an a priori lower bound $\sev^2 \ge \const_{K,a,b} > 0$, which depends only on the data of the problem, but requires a full rank assumption on $K$. By~\eqref{eq: feasible alphas}, any $\omega \in (1, 1 + \delta_{K,a,b})$ then satisfies $\rho_\vartheta(\omega) < \vartheta^2$. Taken together this yields the following result.

\begin{satz}\label{thm: guaranteed acceleration}
Assume $\rank(K) = \min(m,n) \ge 2$. For any $1 < \omega < 1 + \sev^2$ the asymptotic local convergence rate of the modified Sinkhorn method~\eqref{eq: modified algorithm} is faster than for the standard Sinkhorn method. For $1 < \omega < \min\left(1 + \const_{K,a,b}, \frac{2}{1 + \Lambda(K)} \right)$ the modified method is both globally convergent and asymptotically faster than the standard method.
\end{satz}

We remark that our derived a priori interval for $\omega$ is usually very small, and hence our result is of rather theoretical interest. In the relevant cases, when $\sev^2$ is close to one, significant acceleration is achieved only when $\omega$ is close to $\omega^{{\text{opt}}}$ (which tends to two for $\sev^2 \to 1$). A possible heuristic to select a nearly optimal relaxation is to approximate the second largest eigenvalue of $M$ based on the current iterate. After a similarity transform, this requires to compute the spectral norm of a symmetric matrix. An even simpler approach, as suggested in~\cite{thibault17}, is to directly estimate $\vartheta^2$, and hence $\omega^{{\text{opt}}}$, by monitoring the convergence rate of the standard Sinkhorn method in terms of a suitable residual. In the final section~\ref{sec: numerical experiments} we include numerical illustrations, which indicate that in certain cases such heuristics can be quite precise already in the initial phase of the algorithm, resulting in the almost optimal convergence rate at almost no additional cost. This confirms that overrelaxation is a simple way to significantly accelerate the Sinkhorn method in cases where it is slow. For completeness, we should mention that alternative approaches for solving problem~\eqref{eq: problem P} and aiming at fast convergence have been proposed based on Newton's method, see, e.g.,~\cite{Knight2013,Brauer2018} and references therein.

The convergence analysis of the Sinkhorn method is usually carried out in a log-domain formulation~\cite{peyre_cuturi2018}. We choose the closely related framework of compositional data space used, e.g., in statistics~\cite{PawlowskyGlahn2015}, which we think could be of independent interest in this context. In this space, which is introduced in the next section, the Sinkhorn algorithm with a positive matrix $K$ reads as a nonlinear fixed point iteration for an essentially contractive iteration function, as is known from the Birkhoff--Hopf theorem. The main results are then presented in Section~\ref{sec: main results}. Let us note that the assumption that $K$ has strictly positive entries is not essential for all of the results. While global convergence of the standard Sinkhorn method to a unique (up to scaling) positive solution $(u,v)$ of~\eqref{eq: problem P} can be shown under several weaker assumptions, most notably when $a = b = \ones$ and $K$ is square, nonnegative and has total support~\cite{SinkhornKnopp1967}, we require the global contractivity of the process in Hilbert metric (which holds for positive $K$) in our proof that global convergence can still be ensured for some $\omega > 1$ (Theorem~\ref{thm: global convergence}). The idea of accelerating convergence by overrelaxation, on the other hand, is very general and the local spectral analysis provided by Theorem~\ref{thm: local rate modified} applies whenever the iteration~\eqref{eq: modified algorithm} is locally well defined around a (positive) fixed point $(u^*,v^*)$ and $\sev^2 < 1$. Correspondingly, Theorem~\ref{thm: estimate for mu} on a lower bound for $\sev^2$ does not require $K$ to be positive. Hence one has guaranteed acceleration of local convergence for $1 < \omega < 1 + \delta_{K,a,b}$ in several scenarios where $K$ is only nonnegative.

\section{Formulation in compositional data space}\label{sec: compositional data space}

The problem~\eqref{eq: problem P} as well as the Sinkhorn algorithm and its modified variant inherit a natural scaling indeterminacy of the variables $u$ and $v$. It can be therefore formulated in a suitable equivalence space. Here we recast the algorithm in the framework of what is called  \emph{compositional data space}; see, e.g.,~\cite{PawlowskyGlahn2015,vidal2016}. To this aim, let 
\[ \mathcal{C}^{m} \coloneqq \reell^{m}_+\slash\sim ,  \]
where 
\[ x \sim x' \quad :\Longleftrightarrow \quad \exists t > 0: \ x = t x'.
\]
The resulting equivalence class of $x$ will be denoted by $\un{x}$. One specifies a vector addition  and a scalar multiplication  on $\mathcal{C}^{m}$ via 
\[
\un x + \un y \coloneqq \un{x \had y}, \qquad 
\gamma \cdot  \un x \coloneqq \un{x^\gamma}, \quad 
\gamma \in \reell,
\]
where $x^\gamma$ has the components $(x_1^\gamma,\dots,x_n^\gamma)$.
 As a result $(\mathcal{C}^{m}, +, \cdot)$ becomes a real vector space of dimension $m-1$. In this space we consider the so called \emph{Hilbert norm} 
\[
\| \un x \|_H \coloneqq \log \max_{i,j} \frac{x_i}{x_j},
\]
turning $(\mathcal{C}^{m},{}+{}, {}\cdot{}, \|\cdot \|_H)$ into a finite dimensional Banach space. Note that this norm on the equivalence classes coincides with  the well-known \emph{Hilbert distance} on the representatives:
\[
 d_H(x,y) = \| \un x - \un y \|_H.
\]
Similarly we construct a Banach space $\mathcal C^n = \reell^n_+ \slash\sim$.
 
The modified Sinkhorn algorithm~\eqref{eq: modified algorithm} can be interpreted as an iteration in the space $\mathcal{C}^m \times \mathcal{C}^n$ and reads
\begin{equation}\label{eq: iteration in cds}
\begin{aligned}
\un{u}_{\ell + 1} &=  (1-\omega) \cdot \un{u}_\ell + \omega \cdot \un a  - \omega \cdot \DK (\un{v}_\ell) , \\
\un v_{\ell + 1} &= (1-\omega) \cdot  \un v_\ell + \omega \cdot \un b - \omega \cdot \DKt (\un{u}_{\ell + 1}),
\end{aligned}
\end{equation}
where $\DK \colon \mathcal{C}^n \to \mathcal{C}^m$ and $\DKt \colon \mathcal{C}^m \to \mathcal{C}^n$ are now the \emph{nonlinear} maps given by
\[
\DK(\un{v}) = \un{Kv}, \qquad \DKt(\un{u}) = \un{K^\transp u}.
\]
The convergence of the standard Sinkhorn algorithm ($\omega = 1$) is based on a famous result of Birkhoff and Hopf on the contractivity of $\DK$ and $\DKt$. To state it, define the quantities 
\begin{equation}
\eta(K) \coloneqq \max\limits_{i,j,k,\ell} \frac{K_{ik}K_{j\ell}}{K_{jk}K_{i\ell}} \quad\mbox{ and }\quad 
\Lip(K) \coloneqq \frac{\sqrt{\eta(K)} - 1}{\sqrt{\eta(K)} + 1}. \label{def_lambda}
 \end{equation}
 Then the following holds; for a proof, see, e.g.,~\cite[Theorems~3.5~\&~6.2]{eveson_nussbaum95}.
 \begin{satz*}[Birkhoff--Hopf]\label{thm: contractivity of K}
 For any $K \in \reell^{m \times n}_{+}$ and $v,v' \in \reell^{n}_+$ let $\Lambda(K)$ be defined as above. Then 
\begin{equation*}
\label{bh}
\sup_{v, v' \in \reell^{m}_+} \frac{d_H(Kv,Kv')}{d_H(v,v')} = \Lip(K).
\end{equation*}
\end{satz*}
 Note that $\Lip(K) = \Lip(K^\transp) < 1$. As a result, both $\DK$ and $\DKt$ are contractive maps in the Hilbert norm with Lipschitz constant $\Lip(K)$, which is also called the \emph{Birkhoff contraction ratio} of~$K$. Based on this, it is not difficult to establish the global convergence of the standard Sinkhorn algorithm in the space $\mathcal C^m \times \mathcal C^n$ at a rate $O(\Lip(K)^2)$.

It is important to emphasize that studying the convergence in $\mathcal C^m \times \mathcal C^n$, that is, convergence of equivalence classes, is sufficient for understanding the method in $\reell^m_+ \times \reell^n_+$. Indeed, a pair $(\un u^*, \un v^*)$  is a fixed point of~\eqref{eq: iteration in cds} if and only if for any choice of representatives $(u^*, v^*)$ there exist $\lambda, \mu$ such that $\lambda u \had K v= a$ and $\mu v \had K^\transp u = b$. From $\ones_m^\transp a = \ones_n^\transp b$ (here $\ones$ denotes a vector of all ones) it follows that $\lambda =\mu$, and hence, e.g. $u^+ \coloneqq  \lambda^{-1/2} u^* $ and $v^+ \coloneqq  \lambda^{-1/2} v^*$ solve the initial problem~\eqref{eq: problem P}, where $\lambda = u^\transp K v$. Moreover, choosing representatives $(u_\ell, v_\ell)$ of the iterates $(\un u_\ell, \un v_\ell)$ such that $\ones_m^\transp u_\ell = \ones_n^\transp v_\ell = 1$, and setting $u_\ell^+ \coloneqq \lambda_\ell^{-\frac 1 2 } u_\ell$, $v_\ell^+ \coloneqq \lambda_\ell^{-\frac 1 2 } v_\ell$ with $ \lambda_\ell= u_\ell^\transp K v_\ell^{}$, yields a sequence which converges exponentially fast to $(u^*,v^*)$.

We now briefly outline how the local convergence analysis for~\eqref{eq: iteration in cds} can be conducted~\cite{thibault17}, leading to Theorem~\ref{thm: local rate modified}. By combining both steps of the iteration~\eqref{eq: iteration in cds} into a nonlinear fixed point iteration 
\(
(\un{u}_{\ell + 1},
\un{v}_{\ell + 1}) = \FF(\un{u}_\ell, \un{v}_\ell)
\)
in the space $\mathcal C^m \times \mathcal C^n$, one finds that its derivative at the fixed point $(\un u^*, \un v^*)$ takes the form
\begin{equation}\label{eq: iteration matrix}
M_\omega \coloneqq (I_{m+n} - \omega \cdot L)^{-1}[(1 - \omega) \cdot I_{m+n} + \omega \cdot U],
\end{equation}
where
\[
L = \begin{pmatrix} 0 & 0 \\ -\DKt'(\un u^*) & 0 \end{pmatrix}, \qquad U = \begin{pmatrix} 0 & -\DK'(\un v^*) \\ 0 & 0 \end{pmatrix}.
\]
Matrices of the form $M_\omega$ are well known as error iteration matrices of block SOR methods for linear systems. The spectral radius of $M_\omega$ can be computed exactly from formula~\eqref{eq: definition r_alpha}, if the spectral radius $\sev$ of $L + U$ is known; see~\cite[Sec.~6.2]{Young71} or~\cite[Thm.~4.27]{Hackbusch16}. The eigenvalues of $L+U$, however, are square roots of the eigenvalues of the composition of derivatives $\DK'(\un v^*) \DKt'(\un u^*)$, which is a linear map on $\mathcal C^m$. It remains to show that the largest eigenvalue of that operator is precisely the second largest eigenvalue of the matrix $M$ in~\eqref{eq: matrix M}. Indeed, by elementary calculations, $M$ is the matrix representation of $\DK'(\un v^*) \DKt'(\un u^*)$ under the isomorphism $u \mapsto \un{\exp(u)}$ between the subspace $\{ u \in \reell^m \colon \ones_m^\transp u  = 0 \} \subseteq \reell^m$ and $\mathcal C^m$.

\section{Main results}\label{sec: main results}

We prove the global convergence of the modified method for a range of values $\omega$ larger than one. 

\begin{satz}\label{thm: global convergence}
Let $\Lambda = \Lambda(K)$ be the Birkhoff contraction ratio of $K$. For $0< \omega < \frac 2 { 1+\Lip}$, the modified Sinkhorn algorithm~\eqref{eq: iteration in cds} converges, for any starting point, to $(\un u^*, \un v^*)$ exponentially fast. 
\end{satz}

\begin{proof}
Starting from \eqref{eq: iteration in cds}, using the triangle inequality and the contractivity of $\DK$  and $\DKt$ provided by the Birkhoff-Hopf theorem,  we obtain 
\begin{equation*}\label{eq: iteration_1st_ineq}
\begin{aligned}
\|\un{u}_{\ell + 1}-\un u^*\|_H &\leq  \abs{1-\omega} \,  \| \un{u}_\ell -u_*\|_H + \omega \Lip \, \|\un v_\ell -\un v^*\|_H , \\
\|\un v_{\ell + 1}-\un v^*\|_H &\leq \abs{1-\omega} \| \, \un v_\ell -\un v^*\|_H + \omega\Lip  \, \|\un{u}_{\ell + 1}-\un u^*\|_H \\
&\leq  \left(\abs{1-\omega}+(\omega \Lip)^2\right) \| \, \un v_\ell -\un v^*\|_H + \omega\Lip \abs{1-\omega} \, \|\un{u}_{\ell}-\un u^*\|_H.
\end{aligned}
\end{equation*}
As a consequence, for $\Delta u_\ell \coloneqq \|\un{u}_{\ell + 1}-\un u^*\|_H$ and $\Delta v_\ell \coloneqq \|\un{v}_{\ell + 1}-\un v^*\|_H$ we obtain  
\[ {\Delta u_{\ell+1} \choose \Delta v_{\ell+1}} \leq T_\omega {\Delta u_{\ell} \choose \Delta v_{\ell}}, \qquad \text{
where}
\quad  
T_\omega = \begin{pmatrix}
                      \abs{1-\omega} & \omega \Lip \\
                      \omega \Lip \abs{1-\omega} & \abs{1-\omega} + (\omega \Lip)^ 2 
                    \end{pmatrix}, \]
                  and the vector inequality is understood entry-wise. Since all involved quantities are non-negative the inequality can be iterated, which gives
                                        \[ {\Delta u_{\ell+1} \choose \Delta v_{\ell+1}} \leq (T_\omega)^{\ell + 1} {\Delta u_{0} \choose \Delta v_{0}}.\]
Hence, to prove exponential convergence it suffices to show that the spectral radius of $T_\omega$ is strictly less than one. Since the spectral radius equals
 \( \abs{1-\omega} + \frac{(\omega \Lip)^2 }{2} + \sqrt{\frac{(\omega \Lip)^4}{4} + (\omega\Lip)^2 \abs{1-\omega}}
 \)
this is the case if and only if $0 < \omega< 2/(1+\Lip)$.
\end{proof}

Next we provide a lower bound for the second largest eigenvalue $\sev^2$ of the matrix $M$ in~\eqref{eq: matrix M}, which by~\eqref{eq: feasible alphas} then yields an interval for $\omega$ such that the modified method has a strictly faster asymptotic convergence rate than the standard Sinkhorn method.

\begin{satz}\label{thm: estimate for mu}
Let $\rank(K) = \min(m,n) \ge 2$ and
\[
\const_1 = \frac{a_{\min}}{b_{\max}} \cdot \frac{1 - b_{\max}}{\left(\frac{\| K \|_\infty}{\sigma_{\min}(K)}\right)^2 - a_{\min}} > 0, \qquad \const_2 = \frac{b_{\min}}{a_{\max}} \cdot \frac{1 - a_{\max}}{\left(\frac{\| K^\transp \|_\infty}{\sigma_{\min}(K)} \right)^2 - b_{\min}} >0,
\]
where $\sigma_{\min}(K)$ is the smallest positive singular value of $K$, $\| K \|_\infty = \max_{\| v \|_\infty = 1} \| K v \|_\infty$, and the subscripts $\min$, $\max$ denote the smallest and largest entry of the corresponding vector. Then it holds
\[
\sev^2 \ge \const_{K,a,b} \coloneqq \begin{cases} \const_1 &\quad \text{if $m > n$,}  \\
\const_2 &\quad \text{if $m < n$,}  \\
\max(\const_1,\const_2) &\quad \text{if $m = n$.}
\end{cases}
\]
\end{satz}

Note that for a positive matrix $\|K\|_\infty > \sigma_{\min}(K)$. Moreover, $a_{\min} \le \frac{1}{m} \le \frac{1}{n} \le b_{\max} < 1$ if $m \ge n$, and vice versa if $m \le n$. Hence $\const_{K,a,b}$ is indeed smaller than one, which is in line with the bound $\sev^2 \le \Lambda(K)^2$.

\begin{proof}
We consider the case $m \ge n$. Instead of matrix $M$ we consider the positive semidefinite matrix
\[
H = \diag \left(\frac{u^*}{a^{1/2}} \right) K \diag \left(\frac{v^* \had v^*}{b}\right) K^\transp \diag \left(\frac{u^*}{a^{1/2}} \right)  \in \reell^{m \times m},
\]
which is obtained from $M$ by a similarity transformation (and using~\eqref{eq: problem P}). Since the dominant eigenvector of $H$ (with eigenvalue one) is $a^{1/2}$, we have
\[
\sev^2 = \max \left\{ \frac{ \langle w, H w \rangle}{\langle w, w \rangle} \colon \langle w, a^{1/2} \rangle = 0 \right\}.
\]
By projecting on the orthogonal complement of $a^{1/2}$, and noting that $\| a^{1/2} \|^2 = 1$, we first rewrite this as
\[
\sev^2 = \max \frac{ \langle w, H w \rangle - \langle w, a^{1/2} \rangle^2}{\langle w, w \rangle - \langle w, a^{1/2} \rangle^2},
\]
where the maximum is taken over all $w$ that are not collinear to $a^{1/2}$. For such $w$ the numerator is always nonnegative and the denominator is positive. Next we substitute
\[
w = a^{-1/2} \had Kv^* \had z
\]
with a new variable $z$. This yields
\[
\sev^2 = \max \frac{ \langle K^\transp z, \frac{v^* \had v^*}{b} \had K^\transp z \rangle - \langle K^\transp z, v^* \rangle^2}{\left\langle z, \frac{Kv^* \had Kv^*}{a} \had z \right\rangle - \langle K^\transp z, v^* \rangle^2},
\]
where the maximum is taken over all $z$ not collinear with $u^*$ (the numerator is then nonnegative and the denominator is positive). To obtain a lower bound, we now evaluate the expression at $z$ satisfying
\[
K^\transp z = e_j
\]
where $e_j$ denotes the $j$-th unit vector. Note that such $z$ exists ($K^\transp$ has full row rank) and is indeed not collinear to $u^*$, since otherwise $K^\transp u^*$ would be collinear with $e_j$, which contradicts $K^\transp u^* \had v^* = b$. Therefore, using this $z$, we get
\[
\sev^2 \ge \frac{\left(\frac{1}{b_{\max}} - 1 \right) (v^*)_j^2 }{ \left\langle z, \frac{Kv^* \had Kv^*}{a} \had z \right\rangle - (v^*)_j^2}.
\]
We can choose $j$ as the position of a largest entry of the vector $v^*$. Then in the denominator
\[
\left\langle z, \frac{Kv^* \had Kv^*}{a} \had z \right\rangle \le \max_i \frac{(Kv^*)_i^2}{a_i} \| z \|^2  \le \frac{ \| K \|_\infty^2 (v^*)_j^2}{a_{\min}} \frac{1}{\sigma_{\min}(K)^2}. 
\]
This leads to the asserted lower bound $\sev^2 \ge \delta_1$.

When $m \le n$, we can simply interchange the roles of $K$ and $K^\transp$, $a$ and $b$, as well as $u^*$ and $v^*$ in this proof to obtain $\sev^2 \ge \const_2$.
\end{proof}

Taken together, Theorem~\ref{thm: global convergence} and Theorem~\ref{thm: estimate for mu} result in Theorem~\ref{thm: guaranteed acceleration}.

\section{Numerical illustration}\label{sec: numerical experiments}

We illustrate the effect of overrelaxation by two numerical experiments related to optimal transport. The first is motivated by an application to color transfer between images~\cite{Ferradans2014}. The matrix $K = K_\varepsilon$ is generated as
\[
K_{ij} = \exp\left(- \frac{ \| x_i  - y_j \|^2}{\varepsilon}\right)
\]
where $x_i, y_j \in \reell^3$ are RGB values (scaled to $[0,1]$) of $m = n = 1000$ randomly sampled pixels in two different color images, respectively.\footnote{The setup follows the \emph{OT for image color adaptation} example from the Python Optimal Transport toolbox~\cite{Flamary2021pot}. The used images \texttt{ocean\_day.jpg} and \texttt{ocean\_sunset.jpg} are contained in the toolbox.} The vectors $a$ and $b$ are chosen as uniform distributions, i.e.~$a = \ones_m/m$ and $b = \ones_n/n$. We choose $\varepsilon = 0.01$. In this scenario the standard Sinkhorn method is reasonably fast, but still can be accelerated using overrelaxation. A typical outcome for different relaxation strategies is shown in Fig.~\ref{fig:plot} left, where we plot for 500 iterations the $\ell_1$-distance $\| P_\ell - P_* \|_1$ between the matrices $P_\ell = \diag(u_\ell) K \diag(v_\ell)$ and a numerical reference solution $P_* = \diag(u^*) K \diag(v^*)$. This error corresponds to the total variation distance of the corresponding transport plan. Even if this quantity (specifically $P_*$) is not available in a practical computation it is a natural measure for the convergence of the method. Besides the standard Sinkhorn method ($\omega = 1$), we run the method with a fixed relaxation $\omega = 1.5$, and with the `optimal' relaxation $\omega^{\text{opt}}$, which is computed via formula~\eqref{eq: omegaopt} from the second largest singular value $\sev$ of matrix $\diag(1/a^{1/2}) P_* \diag(1/b^{1/2})$ (then $\sev^2$ is the second largest eigenvalue of~\eqref{eq: matrix M}). We do not consider relaxation based on the lower bound on $\sev^2$ in Theorem~\ref{thm: estimate for mu}, since the resulting $\omega$ is too close to one. In all variants of the algorithm the same (uniformly) random starting vectors $u_0$ and $v_0$ are used.

\begin{figure}
    \centering
    \includegraphics[width = .99\textwidth]{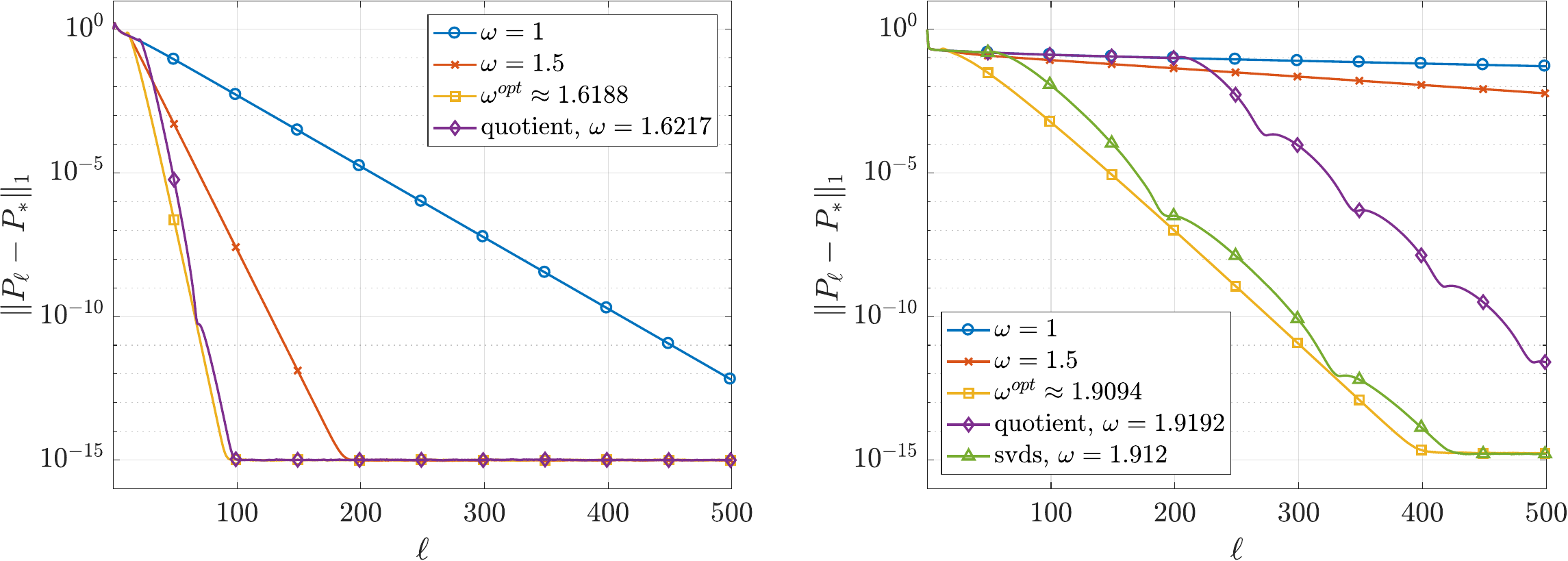}
    \caption{Effect of different relaxation strategies in two examples.}
    \label{fig:plot}
\end{figure}

As can be seen, using $\omega^{\text{opt}}$ significantly accelerates the convergence speed. Moreover, although $\omega^{\text{opt}}$ only provides the optimal local rate, the positive effect shows quite immediately. However, the value of $\omega^{\text{opt}}$ is a priori unknown in practice. Therefore we also tested a simple heuristic, similar to one suggested in~\cite{thibault17}. It is known that the convergence of the Sinkhorn method can be monitored, e.g., through the error $\| P_{\ell} \ones_n - a \|_1$; cf.~\cite[Remark~4.14]{peyre_cuturi2018}. Therefore, since $\vartheta^2$ equals the asymptotic convergence rate of the standard Sinkhorn method, we may take 
\[
\hat \vartheta^2 = \sqrt{\frac{\| P_{\ell} \ones_n - a \|_1}{\| P_{\ell-2} \ones_n - a \|_1}}
\]
as a current approximation for $\vartheta^2$. In the purple curve (diamond markers) in Fig.~\ref{fig:plot} left, we updated $\omega$ a single time after 20 steps of the standard method based on this quantity, and using formula~\eqref{eq: omegaopt}. This comes at almost no additional cost, but yields the near optimal rate in this example. Of course such a heuristic could be applied in a more systematic way, e.g., by monitoring the changes of $\left(\frac{\| P_{\ell} \ones_n - a \|_1}{\| P_{\ell-p} \ones_n - a \|_1} \right)^{1/p}$ for a suitable value of $p$ over several iterations. We note that adapting $\omega$ in (linear and nonlinear) SOR methods based on currently observed convergence rates is a classical idea and has been proposed,~e.g., in~\cite{Young71} or~\cite{Hageman1975}.

As a second example we consider a 1D transport problem between two random measures $a$ and $b$ (generated from a uniform distribution) on an equidistant grid in $[0,1]$, and with $\ell_1$-norm as a cost. The matrix $K$ in this case is given as
\[
K_{ij} = \exp\left(- \frac{\abs{\frac{i}{m-1} - \frac{j}{n-1}}}{\varepsilon}\right).
\]
Again we choose $m=n = 1000$ and $\varepsilon = 0.01$, and then compare different relaxation strategies, but starting from the same random intitialization $(u_0,v_0)$. As can be seen in Fig.~\ref{fig:plot} right, which shows 500 iterations with different relaxation strategies, this problems seems to be more difficult and the standard Sinkhorn method is extremely slow. A suitable relaxation compensates this and restores fast convergence, however, as illustrated by the slow convergence of the curve for $\omega = 1.5$, the estimation of $\omega^{\text{opt}}$, and hence of $\sev^2$, needs to be rather precise. Since here the convergence rate of the standard method stabilizes later, we apply the above heuristic of estimating $\sev^2$ only after 200 iterations of the standard iteration, resulting in the purple curve (diamond markers). The oscillatory behavior occurs because $\omega$ is estimated larger than $\omega^{\text{opt}}$, in which case the spectral radius $\omega - 1$ of the linearized iteration matrix $M_\omega$ in~\eqref{eq: iteration matrix} is achieved at complex eigenvalues. It is possible in this example to update $\omega$ earlier using computationally more expensive heuristics. For instance, the green curve (triangle markers) is obtained by computing after 50 iterations of the standard method an approximation of $\vartheta$ as the second largest singular value of the matrix $\diag(1/a^{1/2}_\ell) P_\ell \diag(1/b^{1/2}_\ell)$, where $a_\ell = u_\ell \had Kv_\ell$ and $b_\ell = v_\ell \had K^\transp u_\ell$. This could be done iteratively, we used the Matlab function \texttt{svds}. This results in an almost optimal convergence rate in this example. Of course, several similar strategies could be devised.

{\small


\begin{thebibliography}{10}

\bibitem{vidal2016}
C.~Barcelo-Vidal and J.-A. Martin-Fernandez.
\newblock The mathematics of compositional analysis.
\newblock {\em Aust. J. Stat.}, 45(4):57--71, 2016.

\bibitem{Brauer2018}
C.~Brauer, C.~Clason, D.~Lorenz, and B.~Wirth.
\newblock A Sinkhorn-Newton method for entropic optimal transport.
\newblock {\em arXiv 1710.06635}, 2017.

\bibitem{Cotar2013}
C.~Cotar, G.~Friesecke, and C.~Kl\"{u}ppelberg.
\newblock Density functional theory and optimal transportation with {C}oulomb
  cost.
\newblock {\em Comm. Pure Appl. Math.}, 66(4):548--599, 2013.

\bibitem{NIPS2013_4927}
M.~Cuturi.
\newblock Sinkhorn distances: Lightspeed computation of optimal transport.
\newblock In C.~J.~C. Burges et al., editors, {\em Advances in Neural Information Processing Systems
  26}, pages 2292--2300. Curran Associates, Inc., 2013.

\bibitem{eveson_nussbaum95}
S.~P. Eveson and R.~D. Nussbaum.
\newblock An elementary proof of the {B}irkhoff-{H}opf theorem.
\newblock {\em Math. Proc. Cambridge Philos. Soc.}, 117(1):31--55, 1995.

\bibitem{Ferradans2014}
S.~Ferradans, N.~Papadakis, G.~Peyr\'{e}, and J.-F. Aujol.
\newblock Regularized discrete optimal transport.
\newblock {\em SIAM J. Imaging Sci.}, 7(3):1853--1882, 2014.

\bibitem{Flamary2021pot}
R.~Flamary {\em et al.}
\newblock POT: Python Optimal Transport.
\newblock {\em J. Mach. Learn. Res.}, 22(78):1--8, 2021.
\newblock{Website: https://pythonot.github.io/}

\bibitem{FranklinLorenz89}
J.~Franklin and J.~Lorenz.
\newblock On the scaling of multidimensional matrices.
\newblock {\em Linear Algebra Appl.}, 114/115:717--735, 1989.

\bibitem{Hackbusch16}
W.~Hackbusch.
\newblock {\em Iterative solution of large sparse systems of equations}.
\newblock Springer, [Cham], second edition, 2016.

\bibitem{Hageman1975}
L.~A. Hageman and T.~A. Porsching.
\newblock{\em Aspects of nonlinear block successive overrelaxation}.
\newblock{\em SIAM J. Numer. Anal.}, 12:316--335, 1975.

\bibitem{Knight2008}
P.~A. Knight.
\newblock The {S}inkhorn-{K}nopp algorithm: convergence and applications.
\newblock {\em SIAM J. Matrix Anal. Appl.}, 30(1):261--275, 2008.

\bibitem{Knight2013}
P.~A. Knight and D.~Ruiz.
\newblock A fast algorithm for matrix balancing.
\newblock {\em IMA J. Numer. Anal.}, 33(3):1029--1047, 2013.

\bibitem{PawlowskyGlahn2015}
V.~Pawlowsky-Glahn, J.~J. Egozcue, and R.~Tolosana-Delgado.
\newblock {\em Modeling and analysis of compositional data}.
\newblock John Wiley \& Sons, Ltd., Chichester, 2015.

\bibitem{Peyreetal2019}
G.~Peyr\'{e}, L.~Chizat, F.-X. Vialard, and J.~Solomon.
\newblock Quantum entropic regularization of matrix-valued optimal transport.
\newblock {\em European J. Appl. Math.}, 30(6):1079--1102, 2019.

\bibitem{peyre_cuturi2018}
G.~Peyr{\'e} and M.~Cuturi.
\newblock Computational optimal transport.
\newblock {\em Found. Trends Mach. Learn.}, 11(5-6):355--607, 2019.

\bibitem{Sinkhorn1967}
R.~Sinkhorn.
\newblock Diagonal equivalence to matrices with prescribed row and column sums.
\newblock {\em Amer. Math. Monthly}, 74:402--405, 1967.

\bibitem{SinkhornKnopp1967}
R.~Sinkhorn and P.~Knopp.
\newblock Concerning nonnegative matrices and doubly stochastic matrices.
\newblock {\em Pacific J. Math.}, 21:343--348, 1967.

\bibitem{thibault17}
A.~Thibault, L.~Chizat, C.~Dossal, and N.~Papadakis.
\newblock Overrelaxed {S}inkhorn-{K}nopp algorithm for regularized optimal transport.
\newblock{\em Algorithms (Basel)}, 14(5), 143, 2021.

\bibitem{Young71}
D.~M. Young.
\newblock {\em Iterative solution of large linear systems}.
\newblock Academic Press, New York-London, 1971.

\end{thebibliography}
}

\end{document}